\documentclass[11pt]{amsart}

\usepackage{amsfonts, amstext, amsmath, amsthm, amscd, amssymb}

\usepackage{graphicx, color}

\usepackage{microtype}

\usepackage[hidelinks,pagebackref,pdftex]{hyperref}

\usepackage{amsfonts}

\hypersetup{
	colorlinks = true,
	urlcolor   = blue,
	citecolor  = blue,
} 
\usepackage{tikz-cd}
\usepackage{verbatim}
\usepackage{float}
\usepackage{makecell}

	\newtheorem{de}{Definition}
	
	\newtheorem{thm}[de]{Theorem}
	\newtheorem{lemma}[de]{Lemma}

	\newtheorem{prop}[de]{Proposition}

\usepackage{amsmath}
\usepackage{amssymb}
\usepackage{graphicx}
\graphicspath{ {Thesis/} }
\usepackage[cmtip,all]{xy}
\usepackage{amsthm}
\usepackage{caption}
\usepackage{subcaption}

\begin{document}

    \title{Twisted Torus Links That Are Unlinks}
    \author{Hong Chang, Thiago de Paiva and Qing Lan}

\address{Hong Chang\\
Beijing International Center for Mathematical Research\\
Peking University\\
Beijing, P.R.China\\
changhong@pku.edu.cn}

\address{Thiago de Paiva\\
Beijing International Center for Mathematical Research\\
Peking University\\
Beijing, P.R.China\\
thhiagodepaiva@gmail.com}

\address{Qing Lan\\
Beijing International Center for Mathematical Research\\
Peking University\\
Beijing, P.R.China\\
lanqing@stu.pku.edu.cn}

\begin{abstract}
A twisted torus link \(T(p,q,r,s)\) is obtained by performing \(s\) full twists on \(r\) adjacent strands of the \((p,q)\)-torus link.
In this paper, we classify twisted torus links that are unlinks.
We give a complete characterization of all parameter families \((p,q,r,s)\) for which the associated twisted torus link is an unlink.
\end{abstract}

\maketitle

\section{Introduction}
Let \(p,q,r,s\) be integers with \(p>1, q\geq 1\), \(2\le r\le p+q\), and \(s\neq0\).
A \emph{twisted torus link} $$T(p,q,r,s)$$ is obtained from the \((p,q)\)-torus link by performing \(s\) full twists on \(r\) adjacent strands.
This construction is illustrated in Figure~\ref{fig:twistedtorus}.
When \(r>p\), the interpretation of the twisting is slightly different and is depicted in Figure~\ref{fig:twistedtorusrlarge}.

Twisted torus links were originally introduced by Dean in his doctoral thesis~\cite{Thesis}, motivated by the problem of determining when Dehn surgery on knots yields Seifert fibered spaces.
Since then, twisted torus links have become an important family of links in low-dimensional topology.
They have been studied from several perspectives, including knot Floer homology~\cite{homology}, bridge spectra~\cite{Bridge}, and Heegaard splittings~\cite{Heegaard}.

Another motivation for studying twisted torus links comes from their role as a tractable subfamily of \emph{generalized T-links}.
T-links, also known as Lorenz links, arise as knotted closed periodic orbits in the Lorenz dynamical system introduced by Lorenz~\cite{Lorenz}.
Birman and Kofman~\cite{BirmanKofman2009} showed that Lorenz links coincide with T-links, which form a combinatorially defined subclass of positive braid links.
By allowing negative twists at the end of the braid, one obtains the family of generalized T-links, which contains twisted torus links as a natural subfamily.
Remarkably, it was proved in~\cite{generalizedtlink} that every link in \(S^3\) is isotopic to a generalized T-link.
As a consequence, the problem of classifying unlinks among generalized T-links (in terms of their combinatorial data) is, in principle, equivalent to classifying all unlinks in \(S^3\).
Within this broader context, twisted torus links provide a natural and accessible testing ground.

The geometric classification of twisted torus links has been studied extensively in the knot case.
In particular, many works have focused on twisted torus knots; see for example
\cite{MR4012126, LeeThiago, tangle, Pai1, MR3801447, twofulltwists, Guntel, nextsimplest, MR2583175, unexpected, dePaiva2023Torus}.
Most notably, Sangyop Lee~\cite{Lee} completely classified twisted torus knots that are unknots, identifying seven explicit families of parameters \((p,q,r,s)\) for which the corresponding twisted torus knot is trivial (see Theorem~\ref{lee}).

In contrast, much less is known in the link case.
For twisted torus links, de Paiva~\cite{Pai1} classified hyperbolic examples with \(|s|>3\), but this represents only an initial step toward a broader geometric classification.
In particular, a complete understanding of which twisted torus links are unlinks has not previously been available.

\begin{figure}[h]
    \centering
\definecolor{linkcolor0}{rgb}{0.85, 0.15, 0.15}
\definecolor{linkcolor1}{rgb}{0.15, 0.15, 0.85}
\begin{minipage}[t]{0.45\textwidth}
        \centering
        \footnotesize
\scalebox{0.6}{
\begin{tikzpicture}[line width=2.5, line cap=round, line join=round]
  \begin{scope}[color=linkcolor0]
    \draw (5.65, 2.85) .. controls (5.65, 2.73) and (5.03, 2.67) .. 
          (4.61, 2.63) .. controls (4.15, 2.59) and (3.56, 2.53) .. 
          (3.25, 2.67) .. controls (2.98, 2.79) and (2.90, 3.11) .. 
          (2.85, 3.41) .. controls (2.77, 3.86) and (2.76, 4.32) .. 
          (2.77, 4.77) .. controls (2.78, 5.10) and (2.79, 5.53) .. 
          (2.96, 5.70) .. controls (3.09, 5.83) and (3.28, 5.82) .. 
          (3.46, 5.76) .. controls (3.66, 5.69) and (3.79, 5.49) .. 
          (3.79, 5.27) .. controls (3.80, 5.09) and (3.83, 4.89) .. 
          (3.99, 4.81) .. controls (4.16, 4.72) and (4.36, 4.72) .. (4.56, 4.73);
    \draw (4.56, 4.73) .. controls (4.93, 4.74) and (5.30, 4.75) .. (5.67, 4.76);
    \draw (5.67, 4.76) .. controls (6.08, 4.78) and (6.49, 4.79) .. (6.90, 4.80);
    \draw (6.90, 4.80) .. controls (7.34, 4.81) and (7.79, 4.83) .. (8.24, 4.84);
    \draw (8.24, 4.84) .. controls (8.44, 4.85) and (8.67, 4.85) .. 
          (8.79, 4.73) .. controls (8.92, 4.61) and (8.93, 4.33) .. (8.94, 4.10);
    \draw (8.96, 3.71) .. controls (9.01, 2.59) and (8.57, 1.44) .. 
          (7.54, 1.18) .. controls (6.42, 0.89) and (5.27, 0.81) .. 
          (4.11, 0.83) .. controls (3.26, 0.84) and (2.18, 0.86) .. 
          (1.64, 1.19) .. controls (1.33, 1.38) and (1.07, 1.65) .. 
          (0.97, 2.00) .. controls (0.83, 2.49) and (0.83, 3.52) .. 
          (0.82, 4.29) .. controls (0.82, 5.03) and (0.81, 6.09) .. 
          (1.06, 6.44) .. controls (1.26, 6.73) and (1.58, 6.90) .. 
          (1.91, 7.03) .. controls (2.52, 7.27) and (3.69, 7.24) .. 
          (4.59, 7.22) .. controls (5.46, 7.20) and (6.67, 7.18) .. 
          (6.77, 6.78) .. controls (6.87, 6.39) and (6.88, 5.60) .. (6.89, 5.00);
    \draw (6.90, 4.60) .. controls (6.90, 4.39) and (6.91, 4.19) .. (6.91, 3.98);
    \draw (6.92, 3.68) .. controls (6.92, 3.60) and (6.92, 3.51) .. (6.92, 3.43);
    \draw (6.92, 3.43) .. controls (6.92, 3.34) and (6.92, 3.25) .. (6.93, 3.15);
    \draw (6.93, 2.83) .. controls (6.94, 2.55) and (6.33, 2.37) .. 
          (5.88, 2.24) .. controls (5.19, 2.03) and (4.48, 1.97) .. 
          (3.77, 2.07) .. controls (3.30, 2.13) and (2.72, 2.22) .. 
          (2.51, 2.53) .. controls (2.30, 2.84) and (2.26, 3.23) .. 
          (2.23, 3.60) .. controls (2.19, 4.05) and (2.16, 4.50) .. 
          (2.17, 4.96) .. controls (2.18, 5.31) and (2.24, 5.66) .. 
          (2.40, 5.98) .. controls (2.54, 6.27) and (2.84, 6.44) .. 
          (3.16, 6.49) .. controls (3.49, 6.54) and (3.83, 6.54) .. 
          (4.14, 6.44) .. controls (4.57, 6.32) and (4.56, 5.55) .. (4.56, 4.93);
    \draw (4.55, 4.53) .. controls (4.55, 4.26) and (4.55, 3.93) .. 
          (4.74, 3.79) .. controls (4.93, 3.66) and (5.34, 3.68) .. (5.66, 3.70);
    \draw (5.66, 3.70) .. controls (6.08, 3.73) and (6.50, 3.75) .. (6.91, 3.78);
    \draw (6.91, 3.78) .. controls (7.35, 3.81) and (7.78, 3.83) .. (8.21, 3.86);
    \draw (8.21, 3.86) .. controls (8.46, 3.88) and (8.70, 3.89) .. (8.95, 3.91);
    \draw (8.95, 3.91) .. controls (9.36, 3.93) and (9.73, 3.65) .. 
          (9.76, 3.25) .. controls (9.81, 2.66) and (9.57, 2.10) .. 
          (9.29, 1.57) .. controls (8.88, 0.82) and (8.04, 0.48) .. 
          (7.18, 0.34) .. controls (5.87, 0.14) and (4.54, 0.18) .. 
          (3.22, 0.29) .. controls (2.11, 0.38) and (0.83, 0.48) .. 
          (0.52, 1.46) .. controls (0.23, 2.36) and (0.19, 3.31) .. 
          (0.17, 4.26) .. controls (0.15, 5.04) and (0.13, 6.07) .. 
          (0.37, 6.58) .. controls (0.50, 6.86) and (0.67, 7.12) .. 
          (0.95, 7.27) .. controls (1.28, 7.45) and (3.18, 7.46) .. 
          (4.33, 7.47) .. controls (5.56, 7.48) and (7.39, 7.49) .. 
          (7.84, 7.11) .. controls (8.29, 6.73) and (8.26, 5.79) .. (8.24, 5.04);
    \draw (8.23, 4.64) .. controls (8.23, 4.45) and (8.22, 4.25) .. (8.22, 4.06);
    \draw (8.21, 3.73) .. controls (8.20, 3.63) and (8.20, 3.53) .. (8.20, 3.43);
    \draw (8.20, 3.43) .. controls (8.20, 3.34) and (8.19, 3.25) .. (8.19, 3.16);
    \draw (8.18, 2.84) .. controls (8.17, 2.55) and (7.94, 2.35) .. 
          (7.72, 2.16) .. controls (7.33, 1.82) and (6.84, 1.64) .. 
          (6.34, 1.51) .. controls (5.52, 1.28) and (4.66, 1.29) .. 
          (3.81, 1.35) .. controls (3.06, 1.39) and (2.21, 1.45) .. 
          (1.82, 2.02) .. controls (1.44, 2.60) and (1.44, 3.48) .. 
          (1.44, 4.26) .. controls (1.44, 4.98) and (1.44, 5.91) .. 
          (1.73, 6.34) .. controls (2.03, 6.77) and (2.60, 6.82) .. 
          (3.12, 6.87) .. controls (3.73, 6.93) and (4.34, 6.97) .. 
          (4.95, 6.93) .. controls (5.70, 6.88) and (5.69, 5.85) .. (5.68, 4.96);
    \draw (5.67, 4.57) .. controls (5.67, 4.34) and (5.67, 4.12) .. (5.67, 3.90);
    \draw (5.66, 3.62) .. controls (5.66, 3.56) and (5.66, 3.50) .. (5.66, 3.43);
    \draw (5.66, 3.43) .. controls (5.66, 3.34) and (5.66, 3.24) .. (5.66, 3.15);
  \end{scope}
  \begin{scope}[color=linkcolor1]
    \draw (5.46, 3.43) .. controls (5.32, 3.43) and (5.18, 3.36) .. 
          (5.18, 3.23) .. controls (5.18, 3.04) and (5.43, 3.03) .. (5.66, 3.03);
    \draw (5.66, 3.03) .. controls (6.08, 3.03) and (6.50, 3.03) .. (6.93, 3.03);
    \draw (6.93, 3.03) .. controls (7.35, 3.03) and (7.77, 3.03) .. (8.19, 3.04);
    \draw (8.19, 3.04) .. controls (8.38, 3.04) and (8.58, 3.07) .. 
          (8.58, 3.24) .. controls (8.58, 3.34) and (8.50, 3.43) .. (8.40, 3.43);
    \draw (8.00, 3.43) .. controls (7.71, 3.43) and (7.41, 3.43) .. (7.12, 3.43);
    \draw (6.72, 3.43) .. controls (6.43, 3.43) and (6.15, 3.43) .. (5.86, 3.43);
  \end{scope}
\end{tikzpicture}

}
 \end{minipage}
    \hfill
\begin{minipage}[t]{0.45\textwidth}
        \centering
        \footnotesize 
        \scalebox{0.6}{ 
\begin{tikzpicture}[line width=1.2, line cap=round, line join=round]
  \begin{scope}[color=linkcolor0]
    \draw (3.84, 2.86) .. controls (3.51, 2.86) and (3.17, 2.71) .. 
          (3.16, 2.41) .. controls (3.15, 2.16) and (3.00, 1.93) .. 
          (2.76, 1.90) .. controls (2.52, 1.87) and (2.28, 1.85) .. 
          (2.03, 1.85) .. controls (1.88, 1.84) and (1.70, 1.83) .. 
          (1.60, 1.91) .. controls (1.50, 1.99) and (1.48, 3.44) .. 
          (1.47, 4.28) .. controls (1.46, 5.15) and (1.45, 6.57) .. 
          (1.56, 6.86) .. controls (1.65, 7.09) and (1.90, 7.17) .. 
          (2.15, 7.21) .. controls (2.50, 7.26) and (2.86, 7.26) .. 
          (3.21, 7.22) .. controls (3.48, 7.18) and (3.77, 7.10) .. 
          (3.75, 6.86) .. controls (3.74, 6.71) and (3.76, 6.55) .. 
          (3.87, 6.46) .. controls (4.00, 6.34) and (4.44, 6.35) .. (4.75, 6.36);
    \draw (4.75, 6.36) .. controls (5.13, 6.37) and (5.50, 6.38) .. (5.88, 6.39);
    \draw (5.88, 6.39) .. controls (6.23, 6.40) and (6.58, 6.41) .. (6.93, 6.42);
    \draw (6.93, 6.42) .. controls (7.33, 6.43) and (7.73, 6.44) .. (8.13, 6.45);
    \draw (8.13, 6.45) .. controls (8.38, 6.45) and (8.64, 6.42) .. 
          (8.86, 6.29) .. controls (9.03, 6.19) and (9.15, 6.01) .. (9.15, 5.82);
    \draw (9.15, 5.63) .. controls (9.15, 4.11) and (9.15, 1.53) .. 
          (8.95, 1.22) .. controls (8.74, 0.90) and (6.43, 0.84) .. 
          (4.99, 0.81) .. controls (3.58, 0.77) and (1.24, 0.71) .. 
          (0.88, 0.92) .. controls (0.52, 1.13) and (0.54, 3.03) .. 
          (0.55, 4.27) .. controls (0.55, 5.44) and (0.57, 7.40) .. 
          (0.78, 7.64) .. controls (0.99, 7.88) and (2.74, 7.85) .. 
          (3.82, 7.83) .. controls (4.86, 7.81) and (6.66, 7.78) .. 
          (6.79, 7.66) .. controls (6.92, 7.54) and (6.92, 6.92) .. (6.93, 6.51);
    \draw (6.93, 6.32) .. controls (6.93, 6.14) and (6.93, 5.96) .. (6.93, 5.78);
    \draw (6.94, 5.59) .. controls (6.94, 5.38) and (6.94, 5.16) .. (6.94, 4.95);
    \draw (6.94, 4.76) .. controls (6.95, 4.46) and (6.95, 4.06) .. 
          (6.80, 3.96) .. controls (6.66, 3.86) and (6.23, 3.85) .. (5.92, 3.85);
    \draw (5.92, 3.85) .. controls (5.51, 3.84) and (5.09, 3.84) .. (4.68, 3.83);
    \draw (4.68, 3.83) .. controls (4.39, 3.83) and (4.07, 3.81) .. 
          (3.95, 3.56) .. controls (3.86, 3.38) and (3.83, 3.17) .. (3.84, 2.96);
    \draw (3.84, 2.77) .. controls (3.86, 2.29) and (3.79, 1.76) .. 
          (3.37, 1.67) .. controls (2.97, 1.60) and (2.56, 1.57) .. 
          (2.16, 1.58) .. controls (1.87, 1.58) and (1.47, 1.58) .. 
          (1.35, 1.69) .. controls (1.24, 1.80) and (1.20, 3.50) .. 
          (1.17, 4.48) .. controls (1.15, 5.47) and (1.11, 7.16) .. 
          (1.24, 7.29) .. controls (1.36, 7.41) and (2.43, 7.40) .. 
          (3.08, 7.40) .. controls (3.91, 7.39) and (4.79, 7.16) .. (4.75, 6.46);
    \draw (4.74, 6.27) .. controls (4.73, 6.11) and (4.79, 5.96) .. 
          (4.90, 5.84) .. controls (5.06, 5.65) and (5.54, 5.66) .. (5.89, 5.66);
    \draw (5.89, 5.66) .. controls (6.24, 5.67) and (6.59, 5.68) .. (6.94, 5.68);
    \draw (6.94, 5.68) .. controls (7.34, 5.69) and (7.74, 5.70) .. (8.15, 5.70);
    \draw (8.15, 5.70) .. controls (8.48, 5.71) and (8.82, 5.72) .. (9.15, 5.72);
    \draw (9.15, 5.72) .. controls (9.82, 5.73) and (9.84, 4.40) .. 
          (9.85, 3.31) .. controls (9.87, 2.31) and (9.89, 0.89) .. 
          (9.45, 0.55) .. controls (9.02, 0.20) and (6.57, 0.16) .. 
          (5.05, 0.13) .. controls (3.53, 0.10) and (1.08, 0.06) .. 
          (0.65, 0.38) .. controls (0.22, 0.70) and (0.17, 2.90) .. 
          (0.14, 4.32) .. controls (0.11, 5.79) and (0.06, 7.94) .. 
          (0.53, 8.06) .. controls (1.00, 8.19) and (3.04, 8.16) .. 
          (4.30, 8.15) .. controls (5.51, 8.13) and (7.61, 8.10) .. 
          (7.85, 7.96) .. controls (8.08, 7.82) and (8.11, 7.07) .. (8.12, 6.54);
    \draw (8.13, 6.35) .. controls (8.13, 6.17) and (8.14, 5.98) .. (8.15, 5.80);
    \draw (8.15, 5.61) .. controls (8.16, 5.39) and (8.17, 5.16) .. 
          (8.02, 5.00) .. controls (7.87, 4.83) and (7.33, 4.85) .. (6.94, 4.85);
    \draw (6.94, 4.85) .. controls (6.60, 4.86) and (6.25, 4.87) .. (5.91, 4.88);
    \draw (5.91, 4.88) .. controls (5.56, 4.89) and (5.11, 4.90) .. 
          (4.89, 4.75) .. controls (4.68, 4.60) and (4.68, 4.23) .. (4.68, 3.93);
    \draw (4.68, 3.74) .. controls (4.68, 3.48) and (4.68, 3.22) .. (4.69, 2.96);
    \draw (4.69, 2.77) .. controls (4.69, 2.14) and (4.69, 1.43) .. 
          (4.13, 1.23) .. controls (3.60, 1.03) and (3.03, 1.02) .. 
          (2.46, 1.03) .. controls (2.01, 1.04) and (1.36, 1.05) .. 
          (1.14, 1.21) .. controls (0.91, 1.36) and (0.88, 3.17) .. 
          (0.85, 4.28) .. controls (0.83, 5.35) and (0.80, 7.19) .. 
          (0.93, 7.42) .. controls (1.07, 7.64) and (2.02, 7.64) .. 
          (2.65, 7.63) .. controls (3.46, 7.63) and (4.26, 7.62) .. 
          (5.05, 7.47) .. controls (5.53, 7.38) and (5.88, 6.97) .. (5.88, 6.48);
    \draw (5.89, 6.30) .. controls (5.89, 6.12) and (5.89, 5.94) .. (5.89, 5.76);
    \draw (5.90, 5.57) .. controls (5.90, 5.37) and (5.90, 5.17) .. (5.91, 4.97);
    \draw (5.91, 4.78) .. controls (5.91, 4.50) and (5.92, 4.22) .. (5.92, 3.94);
    \draw (5.92, 3.75) .. controls (5.93, 3.45) and (5.93, 3.08) .. 
          (5.71, 2.97) .. controls (5.48, 2.86) and (5.03, 2.86) .. (4.69, 2.86);
    \draw (4.69, 2.86) .. controls (4.41, 2.86) and (4.12, 2.86) .. (3.84, 2.86);
  \end{scope}
\end{tikzpicture}
}
\end{minipage}
\caption{Left hand side: The red $(5,2)$-torus knot (with clockwise orientation). By convention in this paper, crossings in this diagram of $T(5,2)$ are considered to be positive, and the braid contained in this diagram of $T(5,2)$ is presented as a product of positive braid group genenrators $\sigma_i$. The blue circle indicates the disk along which the full twist is performed. Right hand side: The twisted torus knot $(5,2,3,-1)$. This is constructed from the $(5,2)$-torus knot by adding $-1$ full twist to the first $3$ strands. } 
    \label{fig:twistedtorus} 
\end{figure}
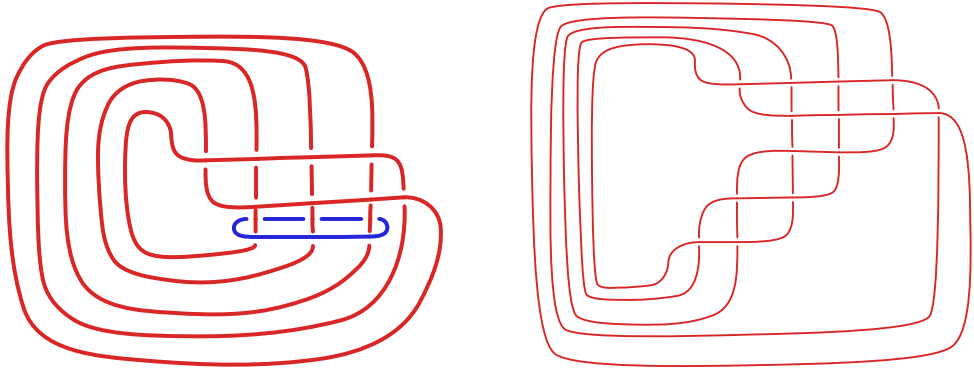

\begin{figure}[h]
    \centering
\definecolor{linkcolor0}{rgb}{0.85, 0.15, 0.15}
\definecolor{linkcolor1}{rgb}{0.15, 0.15, 0.85}
\scalebox{0.6}{ 
\begin{tikzpicture}[line width=2.5, line cap=round, line join=round]
  \begin{scope}[color=linkcolor0]
    \draw (5.63, 3.33) .. controls (5.63, 2.82) and (5.07, 2.54) .. 
          (4.50, 2.50) .. controls (4.02, 2.46) and (3.36, 2.41) .. 
          (3.08, 2.53) .. controls (2.82, 2.65) and (2.74, 2.95) .. 
          (2.69, 3.23) .. controls (2.62, 3.66) and (2.60, 4.09) .. 
          (2.61, 4.53) .. controls (2.62, 4.84) and (2.63, 5.25) .. 
          (2.80, 5.40) .. controls (2.92, 5.52) and (3.10, 5.51) .. 
          (3.27, 5.46) .. controls (3.46, 5.39) and (3.57, 5.20) .. 
          (3.58, 5.00) .. controls (3.59, 4.82) and (3.61, 4.63) .. 
          (3.77, 4.56) .. controls (3.93, 4.48) and (4.12, 4.48) .. (4.31, 4.48);
    \draw (4.31, 4.48) .. controls (4.44, 4.49) and (4.57, 4.49) .. (4.70, 4.49);
    \draw (4.70, 4.49) .. controls (5.01, 4.50) and (5.31, 4.51) .. (5.62, 4.52);
    \draw (5.62, 4.52) .. controls (5.76, 4.52) and (5.91, 4.52) .. (6.05, 4.53);
    \draw (6.05, 4.53) .. controls (6.31, 4.53) and (6.57, 4.54) .. (6.83, 4.55);
    \draw (6.83, 4.55) .. controls (7.16, 4.56) and (7.48, 4.56) .. (7.81, 4.57);
    \draw (7.81, 4.57) .. controls (7.98, 4.58) and (8.16, 4.57) .. 
          (8.30, 4.48) .. controls (8.45, 4.37) and (8.46, 4.11) .. (8.47, 3.90);
    \draw (8.49, 3.50) .. controls (8.50, 3.31) and (8.51, 3.12) .. (8.52, 2.93);
    \draw (8.52, 2.93) .. controls (8.52, 2.77) and (8.53, 2.60) .. (8.54, 2.43);
    \draw (8.55, 2.04) .. controls (8.58, 1.47) and (7.80, 1.28) .. 
          (7.15, 1.12) .. controls (6.09, 0.85) and (4.99, 0.77) .. 
          (3.89, 0.79) .. controls (3.09, 0.81) and (2.07, 0.82) .. 
          (1.56, 1.13) .. controls (1.26, 1.30) and (1.01, 1.56) .. 
          (0.91, 1.89) .. controls (0.77, 2.35) and (0.77, 3.33) .. 
          (0.76, 4.06) .. controls (0.76, 4.77) and (0.75, 5.77) .. 
          (0.99, 6.10) .. controls (1.19, 6.38) and (1.49, 6.54) .. 
          (1.81, 6.67) .. controls (2.38, 6.90) and (3.49, 6.87) .. 
          (4.34, 6.85) .. controls (5.15, 6.84) and (6.30, 6.81) .. 
          (6.56, 6.45) .. controls (6.81, 6.09) and (6.82, 5.34) .. (6.83, 4.74);
    \draw (6.84, 4.38) .. controls (6.84, 4.25) and (6.84, 4.12) .. (6.84, 3.99);
    \draw (6.84, 3.99) .. controls (6.84, 3.90) and (6.84, 3.81) .. (6.84, 3.72);
    \draw (6.85, 3.47) .. controls (6.85, 3.41) and (6.85, 3.35) .. (6.85, 3.29);
    \draw (6.85, 2.96) .. controls (6.86, 2.48) and (6.24, 2.30) .. 
          (5.70, 2.15) .. controls (5.01, 1.95) and (4.28, 1.86) .. 
          (3.56, 1.96) .. controls (3.12, 2.02) and (2.57, 2.10) .. 
          (2.37, 2.40) .. controls (2.18, 2.70) and (2.14, 3.06) .. 
          (2.11, 3.41) .. controls (2.07, 3.84) and (2.05, 4.27) .. 
          (2.06, 4.69) .. controls (2.06, 5.03) and (2.12, 5.37) .. 
          (2.27, 5.68) .. controls (2.40, 5.95) and (2.68, 6.11) .. 
          (2.99, 6.16) .. controls (3.30, 6.20) and (3.61, 6.20) .. 
          (3.91, 6.11) .. controls (4.18, 6.04) and (4.32, 5.75) .. (4.32, 5.46);
    \draw (4.32, 5.46) .. controls (4.31, 5.30) and (4.31, 5.13) .. (4.31, 4.96);
    \draw (4.31, 4.68) .. controls (4.31, 4.64) and (4.31, 4.61) .. (4.31, 4.57);
    \draw (4.31, 4.29) .. controls (4.30, 4.03) and (4.30, 3.72) .. 
          (4.49, 3.60) .. controls (4.68, 3.47) and (5.23, 3.51) .. (5.63, 3.53);
    \draw (5.63, 3.53) .. controls (5.81, 3.54) and (6.00, 3.55) .. (6.18, 3.56);
    \draw (6.18, 3.56) .. controls (6.40, 3.58) and (6.62, 3.59) .. (6.84, 3.60);
    \draw (6.84, 3.60) .. controls (6.99, 3.61) and (7.14, 3.62) .. (7.28, 3.63);
    \draw (7.28, 3.63) .. controls (7.46, 3.64) and (7.64, 3.65) .. (7.81, 3.66);
    \draw (7.81, 3.66) .. controls (8.04, 3.68) and (8.26, 3.69) .. (8.48, 3.70);
    \draw (8.48, 3.70) .. controls (8.97, 3.73) and (9.20, 3.16) .. (9.20, 2.59);
    \draw (9.20, 2.59) .. controls (9.20, 2.42) and (9.20, 2.25) .. (9.20, 2.08);
    \draw (9.20, 1.69) .. controls (9.20, 1.29) and (9.11, 0.86) .. 
          (8.76, 0.73) .. controls (8.31, 0.55) and (7.46, 0.42) .. 
          (6.80, 0.32) .. controls (5.56, 0.13) and (4.29, 0.17) .. 
          (3.04, 0.27) .. controls (1.99, 0.35) and (0.77, 0.45) .. 
          (0.48, 1.38) .. controls (0.21, 2.23) and (0.17, 3.13) .. 
          (0.16, 4.03) .. controls (0.14, 4.77) and (0.13, 5.75) .. 
          (0.35, 6.23) .. controls (0.47, 6.49) and (0.63, 6.75) .. 
          (0.89, 6.89) .. controls (1.21, 7.06) and (3.01, 7.08) .. 
          (4.10, 7.08) .. controls (5.26, 7.09) and (6.99, 7.10) .. 
          (7.40, 6.72) .. controls (7.80, 6.33) and (7.81, 5.47) .. (7.81, 4.77);
    \draw (7.81, 4.37) .. controls (7.81, 4.20) and (7.81, 4.03) .. (7.81, 3.86);
    \draw (7.81, 3.55) .. controls (7.82, 3.46) and (7.82, 3.37) .. (7.82, 3.28);
    \draw (7.82, 3.28) .. controls (7.82, 3.13) and (7.82, 2.98) .. (7.82, 2.82);
    \draw (7.82, 2.43) .. controls (7.82, 2.12) and (7.63, 1.85) .. 
          (7.35, 1.73) .. controls (6.92, 1.55) and (6.46, 1.45) .. 
          (6.00, 1.36) .. controls (5.21, 1.20) and (4.40, 1.22) .. 
          (3.60, 1.27) .. controls (2.89, 1.32) and (2.09, 1.37) .. 
          (1.72, 1.92) .. controls (1.36, 2.47) and (1.36, 3.30) .. 
          (1.36, 4.04) .. controls (1.36, 4.72) and (1.36, 5.60) .. 
          (1.64, 6.01) .. controls (1.92, 6.41) and (2.46, 6.47) .. 
          (2.95, 6.51) .. controls (3.57, 6.57) and (4.19, 6.57) .. 
          (4.80, 6.48) .. controls (5.48, 6.37) and (5.61, 5.54) .. (5.62, 4.78);
    \draw (5.62, 4.78) .. controls (5.62, 4.72) and (5.62, 4.66) .. (5.62, 4.60);
    \draw (5.62, 4.33) .. controls (5.62, 4.25) and (5.62, 4.17) .. (5.62, 4.09);
    \draw (5.62, 3.79) .. controls (5.62, 3.74) and (5.63, 3.69) .. (5.63, 3.64);
  \end{scope}
  \begin{scope}[color=linkcolor1]
    \draw (4.13, 5.53) .. controls (4.01, 5.58) and (3.88, 5.48) .. 
          (3.88, 5.34) .. controls (3.88, 5.09) and (4.10, 4.91) .. (4.31, 4.76);
    \draw (4.31, 4.76) .. controls (4.40, 4.70) and (4.49, 4.64) .. (4.58, 4.58);
    \draw (4.86, 4.38) .. controls (5.08, 4.23) and (5.37, 4.06) .. (5.62, 3.90);
    \draw (5.62, 3.90) .. controls (5.75, 3.82) and (5.89, 3.74) .. (6.02, 3.67);
    \draw (6.35, 3.46) .. controls (6.52, 3.36) and (6.68, 3.26) .. (6.85, 3.16);
    \draw (6.85, 3.16) .. controls (7.08, 3.02) and (7.49, 2.80) .. (7.82, 2.63);
    \draw (7.82, 2.63) .. controls (8.06, 2.50) and (8.30, 2.37) .. (8.55, 2.24);
    \draw (8.55, 2.24) .. controls (8.76, 2.12) and (8.98, 2.00) .. (9.20, 1.88);
    \draw (9.20, 1.88) .. controls (9.49, 1.73) and (9.81, 1.71) .. 
          (9.79, 1.93) .. controls (9.78, 2.18) and (9.61, 2.38) .. (9.38, 2.50);
    \draw (9.03, 2.67) .. controls (8.92, 2.73) and (8.80, 2.79) .. (8.69, 2.84);
    \draw (8.34, 3.02) .. controls (8.22, 3.08) and (8.11, 3.13) .. (7.99, 3.19);
    \draw (7.66, 3.36) .. controls (7.58, 3.40) and (7.51, 3.45) .. (7.44, 3.51);
    \draw (7.15, 3.74) .. controls (7.09, 3.78) and (7.03, 3.83) .. (6.97, 3.88);
    \draw (6.69, 4.11) .. controls (6.54, 4.23) and (6.38, 4.33) .. (6.22, 4.43);
    \draw (5.92, 4.60) .. controls (5.86, 4.64) and (5.81, 4.67) .. (5.75, 4.70);
    \draw (5.45, 4.88) .. controls (5.14, 5.06) and (4.69, 5.32) .. (4.50, 5.39);
  \end{scope}
\end{tikzpicture}

}
\caption{The twisted torus knots $(5,2,6,s)$ are constructed from the $(5,2)$-torus knot by performing full twists along the disk enclosed by the blue circle in this figure. Exactly $6$ red strands pass through this disk. }
\label{fig:twistedtorusrlarge}
\end{figure}
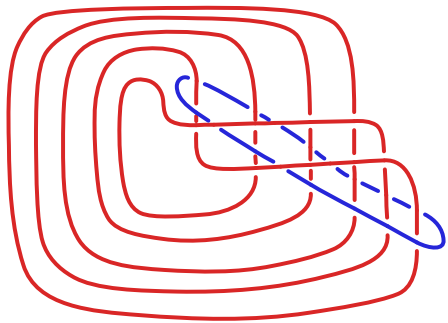

In this paper, we address this gap by classifying unlinks within the family of twisted torus links.
Extending Lee’s work on unknotted twisted torus knots, we obtain a complete characterization of twisted torus links that are unlinks.
Our main result is the following.

\begin{thm}\label{unlink}
Let \(p,q,r,s\) be integers such that \(\gcd(p,q)>1\), \(1\le q<p\), \(2\le r\le p+q\), and \(s\neq0\).
Then the twisted torus link \(T(p,q,r,s)\) is an unlink if and only if one of the following holds:
\begin{enumerate}
\item \((p,q,r,s)=(2n+2,\,2n,\,2n+1,\,-1)\) for some integer \(n\ge1\);
\item \((p,q,r,s)=(4n,\,n,\,2n,\,-1)\) for some integer \(n\ge2\);
\item \((p,q,r,s)=(mn,\,m,\,m,\,-n)\) for some integers \(m\ge2\) and \(n\ge2\).
\end{enumerate}
\end{thm}

\noindent
We remark that the families appearing in Theorem~\ref{unlink} are pairwise disjoint.

Theorem~\ref{unlink} shows that, despite the apparent flexibility in the definition of twisted
torus links, the unlink occurs only in a very restricted set of parameter families. This result may be viewed as a natural extension of Lee’s classification of unknotted
twisted torus knots to the link setting.

From the perspective of generalized T-links, Theorem~\ref{unlink} provides a complete
classification of unlinks within the subfamily of generalized T-links determined by two
pairs of parameters.

The paper is organized as follows.
Section~2 collects preliminary material and establishes notation.
In Section~3 we classify twisted torus links with exactly two components that are unlinks.
Section~4 treats the case of links with more than two components. Section~5 completes the proof of Theorem~\ref{unlink}.

\textbf{Acknowledgements. }The authors thank Prof. Yi Liu and Prof. William Menasco for valuable discussions and suggestions.

\section{Preliminaries}

In this section we collect basic facts and conventions concerning twisted torus links that will be used throughout the paper.
We begin by recalling elementary properties of twisted torus links, including their number of components and convenient braid presentations.
We also record some standard equivalences and constructions that allow us to reduce to simpler parameter ranges.
Finally, we recall a classification result of Lee characterizing when a twisted torus knot is an unknot, which will play a central role in our later arguments.

The twisted torus link \(T(p,q,r,s)\) has \(\gcd(p,q)\) components.
Indeed, by definition, \(T(p,q,r,s)\) is obtained from the torus link \(T(p,q)\) by performing \(s\) full twists on \(r\) of its strands.
Since full twists do not change the number of link components, the resulting link has the same number of components as \(T(p,q)\), namely \(\gcd(p,q)\). Note that this is also true for $p<r\le p+q$.

If \(r \le p\), the twisted torus link \(T(p,q,r,s)\) admits a convenient braid description.
For \(s>0\), it is represented by the braid
\[
(\sigma_1\sigma_2\cdots\sigma_{p-1})^{q}
(\sigma_1\sigma_2\cdots\sigma_{r-1})^{sr},
\]
while for \(s<0\), it is represented by
\[
(\sigma_1\sigma_2\cdots\sigma_{p-1})^{q}
(\sigma_{r-1}^{-1}\sigma_{r-2}^{-1}\cdots\sigma_1^{-1})^{-sr}.
\]

Remark: in our paper we are using a non-conventional orientation for crossings. See the captions in Figure~\ref{fig:twistedtorus}.

\begin{lemma}\label{lemma1}
The twisted torus links \(T(p,q,r,s)\) and \(T(q,p,r,s)\) are equivalent.
\end{lemma}

\begin{proof}
This follows from \cite[Proposition~2.2]{Pai1} or \cite[Lemma~5]{Knottypes}.
\end{proof}

By Lemma~\ref{lemma1}, we may assume throughout that \(q\le p\) when working with a twisted torus link \(T(p,q,r,s)\).

\medskip
As a preliminary observation, we describe a family of torus links that arise naturally as twisted torus links.
This fact was established in \cite{cable}; for completeness, we include a short proof.

\begin{lemma}\label{lemma:torus link in twisted torus link}
The twisted torus link \(T(p,q,q,s)\) is the torus link \(T(q,p+qs)\).
\end{lemma}

\begin{proof}
By Lemma~\ref{lemma1}, the twisted torus link \(T(p,q,q,s)\) is equivalent to \(T(q,p,q,s)\).
A braid representative for \(T(q,p,q,s)\) is
\[
(\sigma_1\sigma_2\cdots\sigma_{q-1})^{p}
(\sigma_1\sigma_2\cdots\sigma_{q-1})^{qs}
=(\sigma_1\sigma_2\cdots\sigma_{q-1})^{p+qs},
\]
which is precisely the standard braid presentation of the torus link \(T(q,p+qs)\).
\end{proof}

\medskip
In \cite{Lee}, Sangyop Lee gave a complete classification of twisted torus knots that are unknots.

\begin{thm}[Theorem~1.1 in \cite{Lee}]\label{lee}
Let \(p,q,r,s\) be integers such that \(p\) and \(q\) are coprime, \(1\le q<p\), \(2\le r\le p+q\), and \(s\neq0\).
The twisted torus knot \(T(p,q,r,s)\) is an unknot if and only if it satisfies one of the following:
\begin{enumerate}
\item \((p,q,r,s)=(n+1,n,n+1,-1)\) for some integer \(n\ge1\);
\item \((p,q,r,s)=(n+1,n,n,-1)\) for some integer \(n\ge2\);
\item \((p,q,r,s)=(mn\pm1,n,n,-m)\) for some integers \(m\ge2\) and \(n\ge2\);
\item \((p,q,r,s)=(n,1,2,-1)\) for some integer \(n\ge3\);
\item \((p,q,r,s)=(4n\pm1,n,2n,-1)\) for some integer \(n\ge2\);
\item \((p,q,r,s)=(n+1,n-1,n,-1)\) for some even integer \(n\ge4\);
\item \((p,q,r,s)=(f_{n+1},f_{n-1},f_n,-1)\) for some integer \(n\ge4\),
where \(f_n\) denotes the \(n\)th Fibonacci number, defined by \(f_1=f_2=1\).
\end{enumerate}
\end{thm}

\section{Case with two components}\label{section:Case with two components}

In this section we classify twisted torus links with exactly two components that are unlinks.
Since \(T(p,q,r,s)\) has \(\gcd(p,q)\) components, the two--component case is precisely the case \(\gcd(p,q)=2\).
Our approach is as follows: assuming \(r\le p\), we compute the linking number of \(T(p,q,r,s)\) in terms of the parameters and impose the condition that it vanishes.
We then apply Lee's classification of unknotted twisted torus knots (Theorem~\ref{lee}) to obtain a finite list of parameter families, show that these families indeed yield unlinks, and finally exclude the remaining range \(r>p\).

Throughout this section we assume \(\gcd(p,q)=2\).
Recall that for an oriented link presented as the closure of a braid,
the linking number between two components is equal to one half of the signed number of crossings
between strands belonging to different components.

\begin{prop}\label{prop:linkingnumber}
Let \(T(p,q,r,s)\) be a twisted torus link with \(r\le p\) and exactly two components (equivalently, \(\gcd(p,q)=2\)).
If the linking number between the two components is zero, then one of the following holds:
\begin{enumerate}
\item \(r\) is even and \(pq=-r^{2}s\). In this case, each component is a twisted torus knot of type \((p/2,q/2,r/2,s)\).
\item \(r\) is odd. Writing \(r=2r'+1\), we have \(pq=-4r'(r'+1)s\). In this case, one component is a twisted torus knot of type \((p/2,q/2,r',s)\) and the other is of type \((p/2,q/2,r'+1,s)\).
\end{enumerate}
\end{prop}

\begin{proof}
Let \(L_1\) and \(L_2\) denote the two components of \(T(p,q,r,s)\), oriented according to the standard braid orientation.
If \(s>0\), then the added twists contribute only positive crossings between the two components, and the \((p,q)\)-torus part also contributes positive inter-component crossings.
Hence \(\operatorname{lk}(L_1,L_2)>0\), so the linking number cannot be zero. Therefore \(s<0\).

Since the link has two components, \(\gcd(p,q)=2\). In particular \(2\mid p\), and in the standard braid model for the \((p,q)\)-torus link each component contains \(p/2\) strands.
The \((p,q)\)-torus part contributes \(\frac{pq}{2}\) signed crossings between the two components, and therefore
\[
\operatorname{lk}_{\mathrm{torus}}=\frac{pq}{4}.
\]

Now consider the twisting portion. Observe that
\[
(\sigma_1\sigma_2\cdots\sigma_{r-1})^{sr}
=\Big((\sigma_1\sigma_2\cdots\sigma_{r-1})^{r}\Big)^{s}
=(\Delta_r^2)^{\,s},
\]
so it is a power of the full twist \(\Delta_r^2\) on the first \(r\) strands. In particular, its underlying permutation is trivial, and hence it does not change which strand indices belong to which component in the closure.
We may therefore identify the two components throughout by odd and even strand indices.

A single full twist \(\Delta_r^2\) makes each pair of strands cross twice.
If among the first \(r\) strands there are \(a\) strands from one component and \(b\) from the other, then one full twist contributes \(2ab\) signed inter-component crossings, and hence contributes \(ab\) to the linking number.
Therefore, the total contribution of \(s\) full twists is \(s\,ab\).

\medskip
\noindent\emph{Case 1: \(r\) even.}
Then \(a=b=r/2\), so the twisting part contributes
\[
\operatorname{lk}_{\mathrm{twist}}=\frac{r^{2}s}{4}.
\]
Hence
\[
0=\operatorname{lk}(L_1,L_2)
=\operatorname{lk}_{\mathrm{torus}}+\operatorname{lk}_{\mathrm{twist}}
=\frac{pq}{4}+\frac{r^{2}s}{4},
\]
which is equivalent to \(pq=-r^{2}s\).
In this situation each component consists of either the odd-indexed or the even-indexed strands, and each is obtained from the \((p/2,q/2)\)-torus braid by inserting \(s\) full twists on the first \(r/2\) strands.
Thus both components are twisted torus knots of type \((p/2,q/2,r/2,s)\).

\medskip
\noindent\emph{Case 2: \(r\) odd.}
Write \(r=2r'+1\).
Among the first \(r\) strands, one component contains \(r'+1\) strands and the other contains \(r'\) strands.
Thus \(a=r'+1\) and \(b=r'\), and the twisting part contributes
\[
\operatorname{lk}_{\mathrm{twist}}=s\,r'(r'+1).
\]
Therefore
\[
0=\operatorname{lk}(L_1,L_2)=\frac{pq}{4}+s\,r'(r'+1),
\]
which is equivalent to \(pq=-4r'(r'+1)s\).
Finally, the component containing \(r'+1\) strands in the twisted block is obtained by inserting the twists on \(r'+1\) strands, while the other is obtained by inserting the twists on \(r'\) strands.
Hence one component is a twisted torus knot \((p/2,q/2,r'+1,s)\) and the other is \((p/2,q/2,r',s)\).
\end{proof}

Combining Proposition~\ref{prop:linkingnumber} with Lee's classification of unknotted twisted torus knots
(Theorem~\ref{lee}), we obtain a complete description of two--component twisted torus links that are unlinks.

\begin{prop}\label{prop:twocomponents}
Let \(p,q,r,s\) be integers satisfying \(\gcd(p,q)=2\), \(1\le q<p\), \(2\le r\le p\), and \(s\neq0\).
Then the twisted torus link \(T(p,q,r,s)\) has exactly two components, each of which is an unknot, and the
linking number between the two components is zero,  if and only if one of the following holds:
\begin{enumerate}
\item \((p,q,r,s)=(2n+2,2n,2n+1,-1)\) for some integer \(n\ge1\);
\item \((p,q,r,s)=(8,2,4,-1)\);
\item \((p,q,r,s)=(2n,2,2,-n)\) for some integer \(n\ge2\).
\end{enumerate}
\end{prop}

\begin{proof}
Assume that \(T(p,q,r,s)\) has two components, both unknotted, and that the linking number between them is zero.
As noted earlier, the linking number of a twisted torus link is positive when \(s>0\); hence we must have \(s<0\).

By Proposition~\ref{prop:linkingnumber}, the vanishing of the linking number imposes strong algebraic
constraints on the parameters and shows that both components are themselves twisted torus knots.
In each of the cases below, these components must therefore appear in the list of unknotted twisted
torus knots given in Theorem~\ref{lee}.

\medskip
\noindent\textbf{Case 1: \(r\ge4\) and \(r\) even.}
Proposition~\ref{prop:linkingnumber} yields
\[
pq=-r^{2}s,
\]
and each component is a twisted torus knot of type \((p/2,q/2,r/2,s)\).
Writing \(p'=p/2\), \(q'=q/2\), and \(r'=r/2\), we obtain
\[
p'q'=-r'^{\,2}s,
\]
and \((p',q',r',s)\) must be an unknot.

Comparing with the seven families listed in Theorem~\ref{lee}, a direct inspection shows that the
only possibility is Case~(4), namely \((n,1,2,-1)\).
This forces \(r'=2\), \(s=-1\), and \(p'q'=4\), hence \((p',q')=(4,1)\).
Therefore
\[
(p,q,r,s)=(8,2,4,-1).
\]

\medskip
\noindent\textbf{Case 2: \(r\ge5\) and \(r\) odd.}
Write \(r=2r'+1\).
By Proposition~\ref{prop:linkingnumber}, the two components are twisted torus knots of types
\[
(p/2,q/2,r'+1,s)\quad\text{and}\quad(p/2,q/2,r',s),
\]
and the vanishing of the linking number gives
\[
pq=-4r'(r'+1)s.
\]
Both components must be unknots.

Among the cases in Theorem~\ref{lee}, the only family with \(s\neq-1\) is Case~(3).
However, it is impossible for both \(r'\) and \(r'+1\) to simultaneously match the third parameter
in that family. Hence \(s=-1\).

Comparing \((p/2,q/2,r',-1)\) with the remaining cases in Theorem~\ref{lee}, the only possibility is
Case~(2), namely \((n+1,n,n,-1)\).
Thus
\[
(p/2,q/2,r')=(n+1,n,n),
\]
and the other component is \((n+1,n,n+1,-1)\), which is also an unknot by Case~(1) of Theorem~\ref{lee}.
Substituting back yields
\[
(p,q,r,s)=(2n+2,2n,2n+1,-1),
\qquad n\ge2.
\]

\medskip
\noindent\textbf{Case 3: \(r=2\).}
Here Proposition~\ref{prop:linkingnumber} gives
\[
pq=-4s,
\]
and each component is a torus knot \((p/2,q/2)\).
For these to be unknots we must have \(q/2=1\), hence \(q=2\).
Writing \(p=2n\) with \(n\ge2\), we obtain \(s=-n\), and therefore
\[
(p,q,r,s)=(2n,2,2,-n).
\]

\medskip
\noindent\textbf{Case 4: \(r=3\).}
Here \(r'=1\), and Proposition~\ref{prop:linkingnumber} gives
\[
pq=-8s.
\]
The two components are a torus knot \((p/2,q/2)\) and a twisted torus knot \((p/2,q/2,2,s)\).
For the torus knot to be an unknot, we must have \(q=2\), hence \(p=-4s\).

The second component is therefore \((-2s,1,2,s)\).
Comparing with Theorem~\ref{lee}, the only admissible case is \(s=-1\), yielding
\[
(p,q,r,s)=(4,2,3,-1),
\]
which corresponds to the family \((2n+2,2n,2n+1,-1)\) with \(n=1\).

\medskip
Combining all cases completes the proof.
\end{proof}

\begin{lemma}\label{lemma2}
The twisted torus links $T(2n+2,2n,2n+1,-1)$ are unlinks for integers $n\ge 1$.
\end{lemma}

\begin{proof}
Let $\sigma_i, i=1, \dots, k-1$ be the generators of the braid group with $k$ strands. For different indices $i, j$ that are not adjacent we have $\sigma_i\sigma_j=\sigma_j\sigma_i$, while $\sigma_i\sigma_{i+1}\sigma_i=\sigma_{i+1}\sigma_i\sigma_{i+1}$. 

By definition the link $T(2n+2,2n,2n+1,-1)$ is the closure of the braid
\[
(\sigma_1\sigma_2 \dots \sigma_{2n+1})^{2n}(\sigma_{2n}^{-1}\sigma_{2n-1}^{-1} \dots\sigma_{1}^{-1})^{2n+1}, 
\]
which is equal to the closure of the braid
\[
(\sigma_{2n}^{-1}\sigma_{2n-1}^{-1} \dots\sigma_{1}^{-1})^{2n+1}(\sigma_1\sigma_2 \dots \sigma_{2n+1})^{2n}. 
\]

We claim the following formula, whose proof we postpone to the next lemma,  Lemma~\ref{lem:braidgrouparg}: For integers $k\ge 0$ and $n\ge 1$ such that $(2n-k-1)\ge 0$, 
\[
\begin{aligned}
&(\sigma_{2n}^{-1}\sigma_{2n-1}^{-1} \dots\sigma_{1}^{-1})^{2n+1}(\sigma_1\sigma_2 \dots \sigma_{2n+1})^{2n}\\
=&(\sigma_{2n}^{-1}\sigma_{2n-1}^{-1} \dots\sigma_{1}^{-1})^{2n-k}\sigma_{2n+1}\sigma_{2n}\dots\sigma_{2n+1-k}(\sigma_1\sigma_2 \dots \sigma_{2n+1})^{2n-k-1}. \\
\end{aligned}
\]

Applying the formula it follows that
\[
\begin{aligned}
&(\sigma_{2n}^{-1}\sigma_{2n-1}^{-1} \dots\sigma_{1}^{-1})^{2n+1}(\sigma_1\sigma_2 \dots \sigma_{2n+1})^{2n}\\
=&(\sigma_{2n}^{-1}\sigma_{2n-1}^{-1} \dots\sigma_{1}^{-1})\sigma_{2n+1}\sigma_{2n}\dots\sigma_{2}. \\
\end{aligned}
\]

The closure of this braid is described as in Figure~\ref{fig:braidword}: the strands with odd indices
\(1,3,5,\dots,2n+1\) form one component, while the strands with even indices
\(2,4,6,\dots,2n+2\) form the other. These two components do not link with each other, and each
component is an unknot.

Hence the closure is the two--component unlink, and therefore
\(T(2n+2,2n,2n+1,-1)\) is an unlink for all \(n\ge1\).
\end{proof}

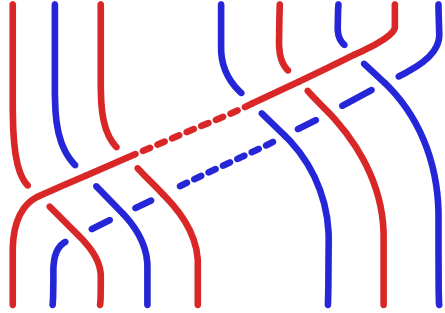
\begin{figure}[h]
    \centering
\definecolor{linkcolor0}{rgb}{0.85, 0.15, 0.15}
\definecolor{linkcolor1}{rgb}{0.15, 0.15, 0.85}
\scalebox{1}{ 
\begin{tikzpicture}[line width=2.5, line cap=round, line join=round]
 
  \begin{scope}[color=linkcolor0]
   
    \draw (4.48, 3.45) .. controls (4.21, 3.33) and (4.13, 3.02) .. 
          (4.13, 2.72) .. controls (4.13, 2.49) and (4.13, 2.15) .. 
        (4.13, 2);
\draw         (4.13, 6) .. controls (4.13, 5.75) and (4.13, 5.12) .. 
          (4.13, 4.77) .. controls (4.13, 4.35) and (4.13, 3.79) .. (4.33, 3.59);

    \draw (4.62, 3.31) .. controls (4.75, 3.18) and (4.88, 3.05) .. (5.01, 2.92);

    \draw (5.01, 2.92) .. controls (5.16, 2.77) and (5.30, 2.59) .. 
         (5.30, 2.38) .. controls (5.30, 2.21) and (5.30, 2.04) .. 
          (5.29, 2);
\draw          (5.30, 6) .. controls (5.30, 5.72) and (5.30, 5.33) .. 
          (5.30, 5.02) .. controls (5.30, 4.69) and (5.30, 4.32) .. (5.51, 4.10);

    \draw (5.79, 3.82) .. controls (5.90, 3.70) and (6.02, 3.59) .. (6.13, 3.47);

    \draw (6.13, 3.47) .. controls (6.38, 3.22) and (6.59, 2.91) .. 
          (6.59, 2.55) .. controls (6.59, 2.23) and (6.59, 2) .. 
          (6.59, 2);
\draw          (7.68, 6) .. controls (7.68, 5.94) and (7.67, 5.72) .. 
         (7.67, 5.50) .. controls (7.67, 5.36) and (7.69, 5.22) .. (7.79, 5.12);

    \draw (8.05, 4.85) .. controls (8.14, 4.75) and (8.24, 4.65) .. (8.33, 4.56);

    \draw (8.33, 4.56) .. controls (8.76, 4.11) and (9.06, 3.54) .. 
          (9.06, 2.91) .. controls (9.06, 2.37) and (9.06, 2) .. 
         (9.06, 2);
\draw          (9.21, 6) .. controls (9.21, 6) and (9.21, 5.82) .. 
         (9.21, 5.70) .. controls (9.21, 5.59) and (8.90, 5.43) .. (8.67, 5.33);

    \draw (8.67, 5.33) .. controls (8.43, 5.21) and (8.18, 5.09) .. (7.93, 4.97);

    \draw (7.93, 4.97) .. controls (7.72, 4.87) and (7.52, 4.78) .. (7.31, 4.68);

    \draw[dashed] (7.31, 4.68) .. controls (7.07, 4.56) and (6.83, 4.46) .. 
          (6.58, 4.36) .. controls (6.27, 4.23) and (5.96, 4.09) .. (5.65, 3.96);

    \draw (5.65, 3.96) .. controls (5.47, 3.88) and (5.28, 3.80) .. (5.10, 3.72);

    \draw (5.10, 3.72) .. controls (4.89, 3.63) and (4.68, 3.54) .. (4.48, 3.45);
   
  \end{scope}

  \begin{scope}[color=linkcolor1]
    
    \draw (4.83, 2.84) .. controls (4.70, 2.77) and (4.67, 2.61) .. 
          (4.67, 2.45) .. controls (4.67, 2.29) and (4.67, 2.12) .. 
          (4.66, 2);
\draw          (4.69, 6) .. controls (4.69, 6) and (4.69, 5.27) .. 
          (4.69, 4.95) .. controls (4.69, 4.56) and (4.69, 4.14) .. (4.96, 3.86);

    \draw (5.24, 3.58) .. controls (5.36, 3.46) and (5.48, 3.33) .. (5.60, 3.21);

    \draw (5.60, 3.21) .. controls (5.78, 3.02) and (5.92, 2.78) .. 
          (5.92, 2.51) .. controls (5.92, 2.25) and (5.92, 2.00) .. 
          (5.92, 2);
\draw          (6.90, 6) .. controls (6.90, 6) and (6.90, 5.62) .. 
          (6.90, 5.41) .. controls (6.90, 5.18) and (7.00, 4.97) .. (7.17, 4.82);

    \draw (7.44, 4.55) .. controls (7.54, 4.45) and (7.65, 4.35) .. (7.75, 4.25);

    \draw (7.75, 4.25) .. controls (8.12, 3.89) and (8.33, 3.39) .. 
          (8.33, 2.87) .. controls (8.33, 2.40) and (8.32, 2) .. 
          (8.32, 2);
\draw          (8.46, 6) .. controls (8.46, 6) and (8.45, 5.84) .. 
          (8.45, 5.69) .. controls (8.45, 5.60) and (8.47, 5.52) .. (8.54, 5.46);

    \draw (8.80, 5.21) .. controls (8.89, 5.13) and (8.98, 5.04) .. (9.08, 4.95);

    \draw (9.08, 4.95) .. controls (9.56, 4.49) and (9.79, 3.84) .. 
          (9.79, 3.18) .. controls (9.79, 2.51) and (9.79, 2) .. 
          (9.80, 2);
\draw         (9.79, 6) .. controls (9.79, 6) and (9.79, 5.88) .. 
          (9.80, 5.60) .. controls (9.80, 5.33) and (9.51, 5.17) .. (9.26, 5.04);

    \draw (8.90, 4.86) .. controls (8.77, 4.79) and (8.64, 4.72) .. (8.51, 4.65);

    \draw (8.16, 4.46) .. controls (8.08, 4.42) and (8.00, 4.38) .. (7.92, 4.34);

    \draw[dashed] (7.59, 4.17) .. controls (7.48, 4.11) and (7.17, 3.97) .. 
          (6.96, 3.87) .. controls (6.74, 3.77) and (6.45, 3.63) .. (6.31, 3.56);

    \draw (5.97, 3.39) .. controls (5.90, 3.36) and (5.83, 3.32) .. (5.76, 3.29);

    \draw (5.42, 3.13) .. controls (5.34, 3.09) and (5.26, 3.05) .. (5.18, 3.01);
    
  \end{scope}
  
\end{tikzpicture}

}
\caption{The braid \((\sigma_{2n}^{-1}\sigma_{2n-1}^{-1} \dots\sigma_{1}^{-1})\sigma_{2n+1}\sigma_{2n}\dots\sigma_{2}\). Its closure is the unlink with two components, given by the red strands and the blue strands respectively. The red component is placed above the blue component, and each component is an unknot. 
}
\label{fig:braidword}
\end{figure}

Now we turn to the proof of the claim in the proof above: 

\begin{lemma}
For integers $k\ge 0$ and $n\ge 1$ such that $(2n-k-1)\ge 0$, we have
\[
\begin{aligned}
&(\sigma_{2n}^{-1}\sigma_{2n-1}^{-1} \dots\sigma_{1}^{-1})^{2n+1}(\sigma_1\sigma_2 \dots \sigma_{2n+1})^{2n}\\
=&(\sigma_{2n}^{-1}\sigma_{2n-1}^{-1} \dots\sigma_{1}^{-1})^{2n-k}\sigma_{2n+1}\sigma_{2n}\dots\sigma_{2n+1-k}(\sigma_1\sigma_2 \dots \sigma_{2n+1})^{2n-k-1}. \\
\end{aligned}
\]
\label{lem:braidgrouparg}
\end{lemma}

\begin{proof}
We prove this claim inductively. For $k=0$ this is clear. In general suppose $k$ and $k+1$ are in the range, and the claim is true for $k$. We would like to prove the case for $k+1$, which is equivalent to the following equality
\[
\begin{aligned}
&(\sigma_{2n}^{-1}\sigma_{2n-1}^{-1} \dots\sigma_{1}^{-1})^{2n-k}\sigma_{2n+1}\sigma_{2n}\dots\sigma_{2n+1-k}(\sigma_1\sigma_2 \dots \sigma_{2n+1})^{2n-k-1} \\
\stackrel{?}{=}&(\sigma_{2n}^{-1}\sigma_{2n-1}^{-1} \dots\sigma_{1}^{-1})^{2n-k-1}\sigma_{2n+1}\sigma_{2n}\dots\sigma_{2n-k}(\sigma_1\sigma_2 \dots \sigma_{2n+1})^{2n-k-2}.  \\
\end{aligned}
\]

Canceling terms appearing on both sides, it suffices to prove
\[
(\sigma_{2n}^{-1}\sigma_{2n-1}^{-1} \dots\sigma_{1}^{-1})\sigma_{2n+1}\sigma_{2n}\dots\sigma_{2n+1-k}(\sigma_1\sigma_2 \dots \sigma_{2n+1})
\stackrel{?}{=} \sigma_{2n+1}\sigma_{2n}\dots\sigma_{2n-k}. 
\]

Note that the left hand side is
\[
\begin{aligned}
&(\sigma_{2n}^{-1}\sigma_{2n-1}^{-1} \dots\sigma_{1}^{-1})\sigma_{2n+1}\sigma_{2n}\dots\sigma_{2n+1-k}(\sigma_1\sigma_2 \dots \sigma_{2n+1})\\
=&(\sigma_{2n}^{-1}\sigma_{2n-1}^{-1} \dots\sigma_{2n-k}^{-1})\sigma_{2n+1}\sigma_{2n}\dots\sigma_{2n+1-k}(\sigma_{2n-k}\sigma_{2n+1-k} \dots \sigma_{2n+1}). \\
\end{aligned}
\]

We claim that this is equal to 
\[
\begin{aligned}
&(\sigma_{2n}^{-1}\sigma_{2n-1}^{-1} \dots\sigma_{2n-k}^{-1})\sigma_{2n-k}\sigma_{2n+1-k}\dots\sigma_{2n}\sigma_{2n+1}\sigma_{2n} \dots \sigma_{2n+1-k}\sigma_{2n-k}\\
=&\sigma_{2n+1}\sigma_{2n} \dots \sigma_{2n+1-k}\sigma_{2n-k}, \\
\end{aligned}
\]
which is just the right hand side. 

It remains to prove the claim, or equivalently
\[
\begin{aligned}
&\sigma_{2n+1}\sigma_{2n}\dots\sigma_{2n+1-k}\sigma_{2n-k}\sigma_{2n+1-k} \dots \sigma_{2n+1}\\
=&\sigma_{2n-k}\sigma_{2n+1-k}\dots\sigma_{2n}\sigma_{2n+1}\sigma_{2n} \dots \sigma_{2n+1-k}\sigma_{2n-k}. \\
\end{aligned}
\]

This is proved using a simple induction of the following form: 
\[
\begin{aligned}
&\sigma_{2n+1} &&\sigma_{2n} &&\sigma_{2n-1} &&\sigma_{2n-2} &&\sigma_{2n-1}  &&\sigma_{2n}  &&\sigma_{2n+1}\\
=&\sigma_{2n+1} &&\sigma_{2n} &&\sigma_{2n-2} &&\sigma_{2n-1} &&\sigma_{2n-2}  &&\sigma_{2n}  &&\sigma_{2n+1}\\
=&\sigma_{2n+1} &&\sigma_{2n-2} &&\sigma_{2n} &&\sigma_{2n-1} &&\sigma_{2n}  &&\sigma_{2n-2}  &&\sigma_{2n+1}\\
=&\sigma_{2n+1} &&\sigma_{2n-2} &&\sigma_{2n-1} &&\sigma_{2n} &&\sigma_{2n-1}  &&\sigma_{2n-2}  &&\sigma_{2n+1}\\
=&\sigma_{2n-2} &&\sigma_{2n-1} &&\sigma_{2n+1} &&\sigma_{2n} &&\sigma_{2n+1} &&\sigma_{2n-1}  &&\sigma_{2n-2} \\
=&\sigma_{2n-2} &&\sigma_{2n-1} &&\sigma_{2n} &&\sigma_{2n+1} &&\sigma_{2n} &&\sigma_{2n-1}  &&\sigma_{2n-2}.  \\
\end{aligned}
\]
\end{proof}

\begin{lemma}\label{lem:824-unlink}
The twisted torus link \(T(8,2,4,-1)\) is the two-component unlink.
\end{lemma}

\begin{proof}
This follows as the special case \(n=2\) of Lemma~\ref{lemma: 4n,n,2n,-1}.
\end{proof}

\begin{lemma}\label{lem:2n22n-unlink}
For every integer \(n\ge2\), the twisted torus link \(T(2n,2,2,-n)\) is the two-component unlink.
\end{lemma}

\begin{proof}
It follows from Lemma~\ref{lem:mnmmn-unlink}.
\end{proof}
         
By the preceding lemmas, each of the cases listed in Proposition~\ref{prop:twocomponents}
indeed yields an unlink.

Now we show that when the number of twisted strands exceeds the
number of strands in the underlying torus link, the resulting twisted torus link can never be an unlink.

\begin{prop}\label{prop:r>p}
Let \(p,q,r,s\) be integers such that \(\gcd(p,q)=d\geq 2\), \(1\le q<p\), \(2\le p<r\le p+q\), and \(s\neq0\).
Then the twisted torus link \(T(p,q,r,s)\) is not an unlink.
\end{prop}

\begin{proof}
Assume first that \(p>2\). The twisted torus link \(L = T(p,q,r,s)\) has \(d = \gcd(p,q)\geq 2\) components.

In the standard braid description of \(L\), each component of the underlying
\((p,q)\)-torus link is represented by exactly \(p/d\) strands.
The twisting block acts on \(r\) adjacent strands, which are distributed among the
\(d\) components. Thus each component \(L_i\) inherits a twisting on
\(r_i'\) strands, where
\[
0\le r_i' \le \frac{p}{d}+\frac{q}{d},
\qquad
r_1'+\cdots+r_d'=r,
\]
and each component \(L_i\) is itself a twisted torus knot of the form
\[
L_i \;=\; T\!\left(\frac{p}{d},\,\frac{q}{d},\,r_i',\,s\right).
\]

We claim that there exists at least one index \(i\) such that
\[
r_i' > \frac{p}{d}.
\]
Indeed, if \(r_i' \le \frac{p}{d}\) for all \(i\), then
\[
r = r_1'+\cdots+r_d' \;\le\; d\cdot \frac{p}{d} \;=\; p,
\]
which contradicts the hypothesis \(r>p\).

Consequently, at least one component of \(L\) is a twisted torus knot
of the form \(T(p/d,q/d,r_i',s)\) with \(r_i'>p/d\).
By Theorem~\ref{lee}, such a twisted torus knot is never an unknot, since in every unknotted
case listed there the third parameter is at most the first.
Thus at least one component of \(T(p,q,r,s)\) is knotted, and the link cannot be an unlink.

Finally, if \(p=2\), then there is no integer \(q\) satisfying \(1\le q<p\) and \(\gcd(p,q)=2\),
so this case does not occur.
\end{proof}

We are now in a position to summarize the results of this section.
Combining the analysis of the vanishing of the linking number, Lee’s classification of unknotted twisted torus knots, and the exclusion of the range \(r>p\), we obtain a complete classification of twisted torus links with exactly two components that are unlinks.

\begin{thm}\label{thm:classification-unlinks}
Let \(p,q,r,s\) be integers with \(\gcd(p,q)=2\), \(1\le q<p\), \(2\le r\le p+q\), and \(s\neq0\).
Then the twisted torus link \(T(p,q,r,s)\) is an unlink if and only if one of the following holds:
\begin{enumerate}
\item \((p,q,r,s)=(2n+2,2n,2n+1,-1)\) for some integer \(n\ge1\);
\item \((p,q,r,s)=(8,2,4,-1)\);
\item \((p,q,r,s)=(2n,2,2,-n)\) for some integer \(n\ge2\).
\end{enumerate}
\end{thm}

\begin{proof}
Let \(p,q,r,s\) satisfy \(\gcd(p,q)=2\), \(1\le q<p\), \(2\le r\le p+q\), and \(s\neq0\).

\medskip
\noindent\emph{(\(\Rightarrow\))}  
Suppose that \(T(p,q,r,s)\) is an unlink. Since \(\gcd(p,q)=2\), the underlying torus link
\(T(p,q)\) has exactly two components, and hence so does \(T(p,q,r,s)\).
In particular, each component is an unknot and the linking number between the two components
is zero.

If \(r\le p\), Proposition~\ref{prop:twocomponents} applies and shows that
\(T(p,q,r,s)\) must belong to one of the following three families:
\[
(2n+2,2n,2n+1,-1),\qquad (8,2,4,-1),\qquad (2n,2,2,-n).
\]

If instead \(p<r\le p+q\), Proposition~\ref{prop:r>p} implies that
\(T(p,q,r,s)\) cannot be an unlink, contradicting our assumption.
Therefore \(r\le p\), and \(T(p,q,r,s)\) must be one of the three families listed above.

\medskip
\noindent\emph{(\(\Leftarrow\))}  
Conversely, suppose that \(T(p,q,r,s)\) belongs to one of the three families in the statement.
By Lemma~\ref{lem:2n22n-unlink}, the links \(T(2n,2,2,-n)\) are unlinks for all \(n\ge2\).
By Lemma~\ref{lem:824-unlink}, the link \(T(8,2,4,-1)\) is an unlink.
Finally, by Lemma~\ref{lemma2}, the family \(T(2n+2,2n,2n+1,-1)\) consists of unlinks.

Thus each of the listed families yields an unlink, completing the proof.
\end{proof}

\section{Cases with more than two components}
\label{section:Case with more than two components}

In this section we study twisted torus links \(T(p,q,r,s)\) with more than two components,
that is, with \(\gcd(p,q)=d>2\).
Our goal is to determine which such links can be unlinks.

As in the two--component case, a fundamental constraint comes from the fact that every
two--component sublink of an unlink must itself be a two--component unlink.
Thus, after choosing any two components of a twisted torus link with \(d>2\),
the resulting sublink must fall into one of the three families classified in
Proposition~\ref{prop:twocomponents}.
This observation severely restricts the possible parameter choices.

We begin by showing that two--component unlink families from
Proposition~\ref{prop:twocomponents} cannot occur simultaneously as sublinks of the same
twisted torus link with zero linking.
We then use this rigidity, together with an analysis of how the twisting parameter \(r\)
interacts with the \(d\) components of the underlying torus link, to obtain a short list
of possible parameter families.
Finally, we verify that these remaining families do indeed give unlinks. This leads to a complete classification of twisted torus unlinks with three or more
components.

\begin{lemma}\label{lem:no-mixing-unlink-families}
No twisted torus link whose two components have linking number \(0\) can contain, as distinct
two--component sublinks, members of two different families from Proposition~\ref{prop:twocomponents},
namely
\[
(2n+2,2n,2n+1,-1),\qquad (8,2,4,-1),\qquad (2n,2,2,-n).
\]
Equivalently, a twisted torus link with zero linking between its components cannot have
two--component sublinks belonging to two different families in Proposition~\ref{prop:twocomponents}.
\end{lemma}

\begin{proof}
Recall that the two components of the links in Proposition~\ref{prop:twocomponents} are as follows:
\begin{itemize}
\item For \((2n+2,2n,2n+1,-1)\), the components are the twisted torus knots
\((n+1,n,n+1,-1)\) and \((n+1,n,n,-1)\).
\item For \((8,2,4,-1)\), both components are \((4,1,2,-1)\).
\item For \((2n,2,2,-n)\), both components are the torus knot \(T(n,1)\) (equivalently, the
twisted torus knot \((n,1,1,-n)\)); in particular, each component has \((p,q)=(n,1)\).
\end{itemize}

Suppose, toward a contradiction, that a twisted torus link \(L\) with two components and zero linking
contains as two--component sublinks representatives from two different families listed above.
Then \(L\) contains components whose \((p/2,q/2)\)-data disagree, since:
\begin{itemize}
\item in the first family, the components have \((p/2,q/2)=(n+1,n)\);
\item in the second family, the components have \((p/2,q/2)=(4,1)\);
\item in the third family, the components have \((p/2,q/2)=(n,1)\).
\end{itemize}
In particular, if \(L\) contained a component of type \((n+1,n,\ast,-1)\) and also a component of type
\((4,1,2,-1)\), then we would have \((n+1,n)=(4,1)\), which is impossible. Likewise, if \(L\) contained
components of types \((4,1,2,-1)\) and \((n,1,\ast,-n)\), then we would need \((4,1)=(n,1)\), hence
\(n=4\), but then the corresponding twisting parameters \(s\) would be \(-1\) and \(-4\), contradicting
that both arise as components of the same zero--linking twisted torus link (whose components share the
same twisting parameter \(s\)). Finally, if \(L\) contained components from the first and third families,
then we would need \((n+1,n)=(n,1)\), which is impossible for \(n\ge1\).

Therefore no such \(L\) can contain two--component sublinks belonging to two different families in
Proposition~\ref{prop:twocomponents}.
\end{proof} 

\begin{prop}\label{morecomponents}
Let \(p,q,r,s\) be integers with \(\gcd(p,q)=d>2\), \(1\le q<p\), \(2\le r\le p\), and \(s\neq0\).
Assume that \((p,q,r,s)\) is an unlink with at least three components.
Then necessarily \(r\ge d\), and \((p,q,r,s)\) must belong to one of the following two families:
\begin{enumerate}
\item \((p,q,r,s)=(4n,n,2n,-1)\) for some integer \(n\ge 3\);
\item \((p,q,r,s)=(mn,m,m,-n)\) for some integers \(m\ge 3\) and \(n\ge 2\).
\end{enumerate}
In particular, no other parameter choice with \(\gcd(p,q)>2\) can yield an unlink with three or more components.
\end{prop}

\begin{proof}
Let \(d=\gcd(p,q)>2\). The underlying torus link \(T(p,q)\) has \(d\) components, each represented (in the standard braid picture) by \(p/d\) strands.

\medskip
\noindent\textbf{Step 1: reduction to \(r\ge d\).}
If \(r<d\), then among the \(d\) components of \(T(p,q)\) there exists at least one component whose strands are disjoint from the block of \(r\) adjacent strands on which we perform the twisting. Take this component and also one component that intersects the block. The two--component sublink determined by these components 
is a torus link \(T(2p/d,\,2q/d)\), whose linking number equals \(\frac{pq}{d^{2}}\neq0\). Hence this two--component sublink is not an unlink, contradicting the assumption that \((p,q,r,s)\) is an unlink. Therefore \(r\ge d\).

\medskip
\noindent\textbf{Step 2: every two--component sublink must be one of the three families in Proposition~\ref{prop:twocomponents}.}
Since \((p,q,r,s)\) is an unlink with at least three components, every two--component sublink is a two--component unlink. Moreover, under the assumption \(r\ge d\), the twisting involves strands from any chosen pair of components, so each two--component sublink is itself a twisted torus link with parameters obtained by restricting to those two components. In particular, after restriction we are in the two--component setting covered by Proposition~\ref{prop:twocomponents}. Hence every two--component sublink must be one of the three unlink families
\[
(2N+2,2N,2N+1,-1),\qquad (8,2,4,-1),\qquad (2N,2,2,-N).
\]

\medskip
\noindent\textbf{Step 3: Exclusion of the family \((2n+2,2n,2n+1,-1)\) for links with at least three components.}

Suppose that a two--component sublink of \((p,q,r,s)\) is of type
\((2n+2,2n,2n+1,-1)\).
The two components of such a link are the twisted torus knots
\((n+1,n,n+1,-1)\) and \((n+1,n,n,-1)\).

Assume first that \((p,q,r,s)\) has exactly three components.
Since every two--component sublink must again be of type \((2n+2,2n,2n+1,-1)\),
each pair of components must consist of one component of type
\((n+1,n,n+1,-1)\) and one of type \((n+1,n,n,-1)\).
This would force each component to be simultaneously of both types,
which is impossible. Hence this family cannot occur when the link has exactly three components.

Now assume that \((p,q,r,s)\) has more than three components.
Then at least two components must be of the same type, either both
\((n+1,n,n+1,-1)\) or both \((n+1,n,n,-1)\).
Consider two components of type \((n+1,n,n+1,-1)\).
The two--component sublink they determine is then the twisted torus link
\((2n+2,2n,2n+2,-1)\).
By Proposition~\ref{prop:linkingnumber}, the linking number of this link is nonzero,
and hence it is not an unlink. 
This contradicts the assumption that every two--component sublink of \((p,q,r,s)\)
is an unlink. Similar argument also applies when there are two \((n+1,n,n,-1)\)s.

Therefore, no unlink with at least three components can have a two--component
sublink of type \((2n+2,2n,2n+1,-1)\).

\medskip
\noindent\textbf{Step 4: the remaining two families force the stated parameter forms.}
By Lemma~\ref{lem:no-mixing-unlink-families}, a twisted torus link with all two--component sublinks unlinks cannot contain two--component sublinks from different families in Proposition~\ref{prop:twocomponents}. Therefore all two--component sublinks must belong to the \emph{same} remaining family.

\smallskip
\noindent\emph{Case A: all two--component sublinks are \((8,2,4,-1)\).}
For any pair of components, the restriction has torus parameters \((2p/d,\,2q/d)\), twist parameter \(s\), and \(r\)-parameter \((2r/d)\). Thus
\[
\frac{2p}{d}=8,\qquad \frac{2q}{d}=2,\qquad \frac{2r}{d}=4,\qquad s=-1.
\]
Hence \(p=4d\), \(q=d\), \(r=2d\), and \(s=-1\), i.e.
\[
(p,q,r,s)=(4n,n,2n,-1)\quad\text{with }n=d\ge3.
\]

\smallskip
\noindent\emph{Case B: all two--component sublinks are \((2N,2,2,-N)\).}
Similarly we obtain
\[
\frac{2q}{d}=2\ \Rightarrow\ q=d,\qquad \frac{2r}{d}=2\ \Rightarrow\ r=d,
\]
and
\[
\frac{2p}{d}=2N\ \Rightarrow\ p=Nd,\qquad s=-N.
\]
Writing \(m=d\) and \(n=N\), we obtain
\[
(p,q,r,s)=(mn,m,m,-n),
\]
with \(m=d\ge3\) and \(n\ge2\).

\medskip
This proves that the only possible parameter families for an unlink with at least three components (and \(\gcd(p,q)>2\)) are those stated in the proposition.
\end{proof}

\begin{figure}[h]
    \centering
\definecolor{linkcolor0}{rgb}{0.85, 0.15, 0.15}
\definecolor{linkcolor1}{rgb}{0.15, 0.15, 0.85}
\begin{tikzpicture}[line width=2.4, line cap=round, line join=round]
  \begin{scope}[color=linkcolor0]
    \draw (3.43, 2.74) -- (3.29, 2.74) -- (3.29, 2.10) -- (2.67, 2.10) -- (2.67, 7.74) -- (3.95, 7.74) -- (3.95, 6.95) -- (4.10, 6.95);
    \draw (4.10, 6.95) -- (4.71, 6.95);
    \draw (4.71, 6.95) -- (4.88, 6.95);
    \draw (4.88, 6.95) -- (6.56, 6.95);
    \draw (6.56, 6.95) -- (6.69, 6.95);
    \draw (6.69, 6.95) -- (8.30, 6.95);
    \draw (8.30, 6.95) -- (8.46, 6.95);
    \draw (8.46, 6.95) -- (9.35, 6.95) -- (9.35, 6.88);
    \draw (9.35, 6.57) -- (9.35, 0.26) -- (0.72, 0.29) -- (0.72, 8.94) -- (8.30, 8.94) -- (8.30, 7.15);
    \draw (8.30, 6.90) -- (8.30, 6.83);
    \draw (8.30, 6.58) -- (8.30, 0.88) -- (1.30, 0.88) -- (1.30, 8.55) -- (6.56, 8.55) -- (6.56, 7.15);
    \draw (6.56, 6.91) -- (6.56, 6.84);
    \draw (6.56, 6.59) -- (6.56, 4.96);
    \draw (6.56, 4.72) -- (6.56, 4.70) -- (4.96, 4.70);
    \draw (4.96, 4.70) -- (4.77, 4.70);
    \draw (4.77, 4.70) -- (4.00, 4.70);
    \draw (4.00, 4.70) -- (3.85, 4.70) -- (3.85, 3.08);
    \draw (3.85, 2.84) -- (3.85, 2.78);
    \draw (3.85, 2.55) -- (3.85, 1.38) -- (2.00, 1.38) -- (2.00, 8.11) -- (4.74, 8.11) -- (4.74, 7.15);
    \draw (4.74, 6.91) -- (4.74, 6.84);
    \draw (4.74, 6.60) -- (4.74, 5.06);
    \draw (4.74, 4.81) -- (4.74, 4.75);
    \draw (4.74, 4.50) -- (4.74, 3.03);
    \draw (4.74, 2.76) -- (4.74, 2.74) -- (4.01, 2.74);
    \draw (4.01, 2.74) -- (3.85, 2.74);
    \draw (3.85, 2.74) -- (3.43, 2.74);
  \end{scope}
  \begin{scope}[color=linkcolor1]
    \draw (4.10, 6.86) -- (4.10, 6.80) -- (4.72, 6.80);
    \draw (4.72, 6.80) -- (4.88, 6.80);
    \draw (4.88, 6.80) -- (6.56, 6.80);
    \draw (6.56, 6.80) -- (6.69, 6.80);
    \draw (6.69, 6.80) -- (8.29, 6.80);
    \draw (8.29, 6.80) -- (8.45, 6.80);
    \draw (8.45, 6.80) -- (9.31, 6.80);
    \draw (9.31, 6.80) -- (9.55, 6.80) -- (9.55, 0.13) -- (0.53, 0.13) -- (0.53, 9.08) -- (8.45, 9.08) -- (8.45, 7.15);
    \draw (8.45, 6.90) -- (8.45, 6.83);
    \draw (8.45, 6.58) -- (8.45, 0.73) -- (1.10, 0.73) -- (1.10, 8.69) -- (6.70, 8.69) -- (6.70, 7.15);
    \draw (6.70, 6.91) -- (6.70, 6.84);
    \draw (6.70, 6.59) -- (6.70, 4.85);
    \draw (6.70, 4.85) -- (4.00, 4.85) -- (4.00, 4.83);
    \draw (4.00, 4.54) -- (4.00, 3.07);
    \draw (4.00, 2.83) -- (4.00, 2.78);
    \draw (4.00, 2.54) -- (4.00, 1.23) -- (1.79, 1.23) -- (1.79, 8.30) -- (4.95, 8.30) -- (4.95, 7.15);
    \draw (4.95, 6.91) -- (4.95, 6.84);
    \draw (4.95, 6.60) -- (4.95, 5.05);
    \draw (4.95, 4.80) -- (4.95, 4.74);
    \draw (4.95, 4.50) -- (4.95, 2.88);
    \draw (4.95, 2.88) -- (4.01, 2.88);
    \draw (4.01, 2.88) -- (3.85, 2.88);
    \draw (3.85, 2.88) -- (3.43, 2.88);
    \draw (3.43, 2.56) -- (3.43, 1.89) -- (2.48, 1.89) -- (2.48, 7.89) -- (4.10, 7.89) -- (4.10, 7.14);
  \end{scope}
\end{tikzpicture}
\caption{A projection of the embedded annulus, with two boundary components colored red and blue. For any $n$, the link $(4n,n,2n,-1)$ is supported on this annulus. } 
    \label{fig:lq} 
\end{figure}
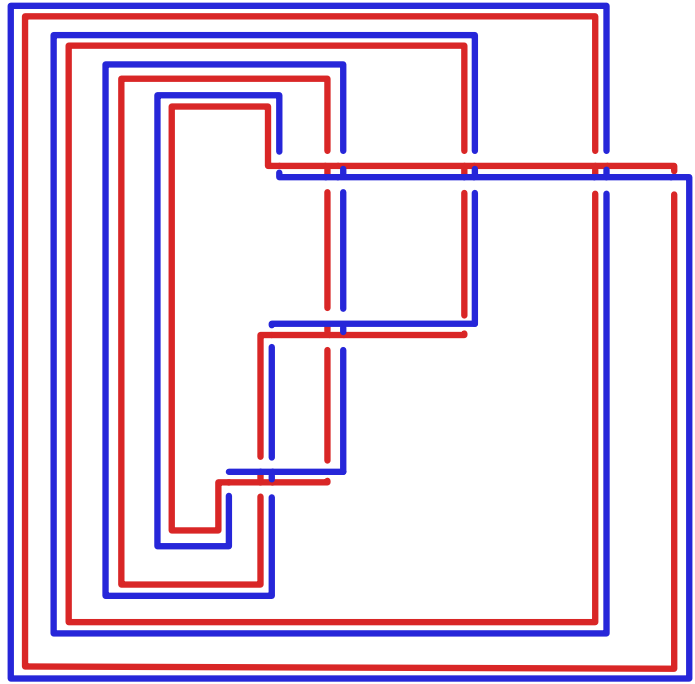

The following lemma admits alternative proofs based on the methods of
\cite[Theorem~1]{cable}, \cite[Theorem~4.3]{de2021satellites}, and
\cite[Theorem~5.4]{de2023hyperbolic}. We include a different proof here for completeness.

\begin{lemma}\label{lemma: 4n,n,2n,-1}
For every integer \(n\ge2\), the twisted torus link \(T(4n,n,2n,-1)\) is the \(n\)-component unlink.
\end{lemma}

\begin{proof}
Consider the embedded annulus \(A\subset S^{3}\) indicated in Figure~\ref{fig:lq}.
For each \(n\ge2\), choose \(n\) distinct numbers \(0<a_1<\cdots<a_n<1\) and set
\[
\Gamma_n \;=\; \bigsqcup_{j=1}^{n}\bigl(S^{1}\times\{a_j\}\bigr)\subset S^{1}\times I .
\]
Under the embedding \(S^{1}\times I\hookrightarrow S^{3}\) represented by Figure~\ref{fig:lq},
the image of \(\Gamma_n\) is precisely the twisted torus link \(T(4n,n,2n,-1)\).

We claim that \(A\) is ambient isotopic in \(S^{3}\)
to a standardly embedded annulus. It suffices to observe that the core curve is an unknot and its framing corresponding to this annulus is zero.  The core curve of \(A\) is the twisted torus knot
\(T(4,1,2,-1)\), which is an unknot by Theorem~\ref{lee}. The framing can be seen to be zero by computing the linking number.  As a result, under such an ambient isotopy, the link \(\Gamma_n\subset A\) is carried
to \(n\) parallel core circles in a standard annulus. These circles are pairwise unlinked and each bounds
an embedded disk in \(S^{3}\); therefore their union is the \(n\)-component unlink.

Consequently, \(T(4n,n,2n,-1)\), being isotopic to the image of \(\Gamma_n\), is the \(n\)-component unlink.
\end{proof}

\begin{lemma}\label{lem:mnmmn-unlink}
For integers \(m,n\ge2\), the twisted torus link \(T(mn,m,m,-n)\) is the unlink.
\end{lemma}

\begin{proof}
By Lemma~\ref{lemma:torus link in twisted torus link}, a twisted torus link of the form
\(T(p,q,q,s)\) is equivalent to the torus link \(T(q,p+qs)\).
Applying this with \(p=mn\), \(q=m\), and \(s=-n\), we obtain
\[
T(mn,m,m,-n)=T\bigl(m,\,mn+m(-n)\bigr)=T(m,0).
\]
The torus link \(T(m,0)\) consists of \(m\) parallel, unlinked components, and hence is the
\(m\)-component unlink.
\end{proof}

\begin{thm}\label{thm2:classification-unlinks}
Let \(p,q,r,s\) be integers such that \(\gcd(p,q)=d>2\), \(1\le q<p\), \(2\le r\le p+q\), and \(s\neq0\).
Then \(T(p,q,r,s)\) is an unlink if and only if one of the following holds:
\begin{enumerate}
\item \((p,q,r,s)=(4n,\,n,\,2n,\,-1)\) for some integer \(n\ge 3\);
\item \((p,q,r,s)=(mn,\,m,\,m,\,-n)\) for some integers \(m\ge 3\) and \(n\ge 2\).
\end{enumerate}
In particular, in case (1) the link has \(d=n\) components, and in case (2) it has \(d=m\) components.
\end{thm}

\begin{proof}
Write \(d=\gcd(p,q)>2\). Then \(T(p,q,r,s)\) has exactly \(d\) components.

\medskip
\noindent\emph{(\(\Rightarrow\))} Suppose that \(T(p,q,r,s)\) is an unlink. By Proposition~\ref{prop:r>p}, $r\leq p$.

Since \(r\le p\), Proposition~\ref{morecomponents} applies: it shows that under the hypotheses
\(\gcd(p,q)>2\), \(2\le r\le p\), and \(s\neq0\), the only possible parameter families for a twisted torus link
with at least three components whose two--component sublinks are all unlinks are
\[
(p,q,r,s)=(4n,n,2n,-1)\quad\text{for some }n\ge3,
\]
or
\[
(p,q,r,s)=(mn,m,m,-n)\quad\text{for some }m\ge3,\ n\ge2.
\]
Therefore \(T(p,q,r,s)\) must be of one of the two types listed in the statement.

\medskip
\noindent\emph{(\(\Leftarrow\))} Conversely, assume that \((p,q,r,s)\) satisfies one of the two conditions.

If \((p,q,r,s)=(4n,n,2n,-1)\) with \(n\ge3\), then \(T(p,q,r,s)\) is an unlink by
Lemma~\ref{lemma: 4n,n,2n,-1}.

If \((p,q,r,s)=(mn,m,m,-n)\) with \(m\ge3\) and \(n\ge2\), then \(T(p,q,r,s)\) is an unlink by
Lemma~\ref{lem:mnmmn-unlink}.

Thus each of the listed families indeed yields an unlink, completing the proof.
\end{proof}

\section{Classification of Twisted Torus Links That Are Unlinks}
\label{section:Case with more than two components}

We now complete the classification of twisted torus links that are unlinks.
By combining the results obtained for the two--component case with those for links having
three or more components, we arrive at a complete characterization.

\begin{thm}
Let \(p,q,r,s\) be integers such that \(\gcd(p,q)>1\), \(1\le q<p\), \(2\le r\le p+q\), and \(s\neq0\).
Then the twisted torus link \(T(p,q,r,s)\) is an unlink if and only if one of the following holds:
\begin{enumerate}
\item \((p,q,r,s)=(2n+2,\,2n,\,2n+1,\,-1)\) for some integer \(n\ge1\);
\item \((p,q,r,s)=(4n,\,n,\,2n,\,-1)\) for some integer \(n\ge2\);
\item \((p,q,r,s)=(mn,\,m,\,m,\,-n)\) for some integers \(m\ge2\) and \(n\ge2\).
\end{enumerate}
\end{thm}

\begin{proof}
Let \(p,q,r,s\) satisfy \(\gcd(p,q)>1\), \(1\le q<p\), \(2\le r\le p+q\), and \(s\neq0\), and set
\(d=\gcd(p,q)\).
Recall that the twisted torus link \(T(p,q,r,s)\) has exactly \(d\) components.

\medskip
\noindent\emph{(\(\Rightarrow\))}  
Suppose that \(T(p,q,r,s)\) is an unlink.

\smallskip
\noindent\textbf{Case 1: \(d=2\).}
In this case \(T(p,q,r,s)\) has exactly two components.
By the classification of two--component twisted torus unlinks
(Theorem~\ref{thm:classification-unlinks}), the parameters must satisfy
\[
(p,q,r,s)=(2n+2,2n,2n+1,-1),\quad (8,2,4,-1),\quad \text{or}\quad (2n,2,2,-n),
\]
for suitable integers \(n\).
The link \((8,2,4,-1)\) is precisely the case \(n=2\) of the family \((4n,n,2n,-1)\),
and the family \((2n,2,2,-n)\) coincides with \((mn,m,m,-n)\) when \(m=2\).
Thus all possibilities with \(d=2\) are contained in the three families listed in the statement.

\smallskip
\noindent\textbf{Case 2: \(d>2\).}
Then \(T(p,q,r,s)\) has at least three components.
In this situation, Theorem~\ref{thm2:classification-unlinks} applies and implies that
\((p,q,r,s)\) must be of one of the following two forms:
\[
(p,q,r,s)=(4n,n,2n,-1)\quad\text{for some }n\ge3,
\]
or
\[
(p,q,r,s)=(mn,m,m,-n)\quad\text{for some }m\ge3,\ n\ge2.
\]
These correspond exactly to items~(2) and~(3) in the statement of the theorem.

\medskip
This proves the forward implication.

\medskip
\noindent\emph{(\(\Leftarrow\))}  
Conversely, if \((p,q,r,s)\) belongs to one of the three families listed above, then
\(T(p,q,r,s)\) is an unlink by the same results cited above.
\end{proof}

\bibliographystyle{acm}
\bibliography{reference.bib}


 \end{document}